\documentclass[a4paper,11pt,reqno,noindent]{amsart}
\usepackage[centertags]{amsmath} 
\usepackage{amsfonts,amssymb,amsthm}
\usepackage{mathtools}
\usepackage[utf8]{inputenc}
\usepackage[foot]{amsaddr}

\usepackage{color}
\definecolor{darkred}{rgb}{0.8,0,0}
\definecolor{darkblue}{rgb}{0,0,0.55}
\definecolor{darkgreen}{rgb}{0,0.39,0}
\usepackage[colorlinks=true]{hyperref} 
\hypersetup{linkcolor=darkblue,citecolor=darkgreen}

\usepackage{graphicx}
\usepackage{dsfont}
\usepackage[english]{babel} \usepackage{newlfont} \usepackage{color}
\usepackage[body={15cm,21.5cm},centering]{geometry} \usepackage{fancyhdr}
 \usepackage{esint}

\usepackage[active]{srcltx}
\usepackage{mathrsfs}

\newtheorem{theorem}{Theorem} \newtheorem{lemma}{Lemma}
\newtheorem{proposition}{Proposition} 
 
\newtheorem{definition}{Definition}

\theoremstyle{definition} 
 
\newtheorem{remark}{Remark}

 \newcommand{\supp}{\operatorname{supp}}

\newcommand{\e}{\varepsilon}

\newcommand\ecke{\mathop{\hbox{\vrule height 7pt width .3pt depth 0pt \vrule
height .3pt width 5pt depth 0pt}}\nolimits}

\newcommand{\R}{\mathbb{R}} 
\newcommand{\N}{\mathbb{N}} 
 
\renewcommand{\d}{\mathrm{d}} \renewcommand{\L}{\mathbb{L}}

 \renewcommand{\H}{\mathcal{H}}
\newcommand{\Ha}{\mathcal{H}}
 
\renewcommand{\L}{{\mathcal L}} \newcommand\wto{\rightharpoonup}

\newcommand\wsto{\stackrel{*}{\rightharpoonup}}

\newcommand{\calL}{\mathcal{L}} 

\newcommand{\calE}{\mathcal{E}}
\newcommand{\calP}{\mathcal{P}}
\newcommand{\calR}{\mathcal{R}}
\newcommand{\calD}{\mathcal{D}} \newcommand{\calH}{\mathcal{H}}

\newcommand{\dx}{\,\mathrm{d}x}
\newcommand{\dy}{\,\mathrm{d}y}

\newcommand{\W}{\mathcal{W}}

\newcommand{\eps}{\varepsilon}

\newcommand{\curr}[1]{[\![#1]\!]}
\newcommand{\ovar}[1]{\underline{v}(#1)}
\newcommand{\angles}[2][]{#1\langle#2#1\rangle}
\newcommand{\mass}[2][]{\mathsf{M}_{#1}(#2)}
\newcommand{\sfI}{\mathsf{\mathbf{I}}}
\newcommand{\sfR}{\mathsf{\mathbf{R}}}
\newcommand{\sfV}{\mathsf{V}}
\newcommand{\sfIV}{\mathsf{IV}}
\newcommand{\sfRV}{\mathsf{RV}}

\newcommand{\weakto}{\rightharpoonup}
\newcommand{\weakstarto}{\stackrel{\ast}{\rightharpoonup}}
\newcommand{\sphere}{\mathbb{S}}

\newcommand{\dwp}{\vartheta}
\begin{document}

\title{Phase separation on varying surfaces and convergence of diffuse interface approximations}
\author{Heiner Olbermann}
\address[Heiner Olbermann]{UCLouvain, Belgium}
\email{heiner.olbermann@uclouvain.be}
\author{Matthias Röger}\address[Matthias Röger]{Department of Mathematics,
Technische Universit\"at Dortmund}
\email{matthias.roeger@tu-dortmund.de}
\maketitle

\date{\today} 

\maketitle
\begin{abstract}
In this paper we study phase separations on generalized hypersurfaces in Euclidian space.
We consider a diffuse surface area (line tension) energy of Modica--Mortola type and prove a compactness and lower bound estimate in the sharp interface limit.
We use the concept of generalized $BV$ functions over currents as introduced by Anzellotti et.~al. [Annali di Matematica Pura ed Applicata, 170, 1996] to give a suitable formulation in the limit and achieve the necessary compactness property.
We also consider an application to phase separated biomembranes where a Willmore energy for the membranes is combined with a generalized line tension energy.
For a diffuse description of such energies we give a lower bound estimate in the sharp interface limit.
\\[2ex]\noindent%
\\[2ex]\noindent%
\end{abstract}

\section{Introduction}
Phase separation processes are ubiquitous in many applications from material sciences, physics and biology.
The mathematical analysis is in many cases well understood and has seen many contributions from different communities.

As two fundamental classes of descriptions we can distinguish \emph{sharp} and \emph{diffuse interface models}.
In the sharp interface approach each material point belongs to exactly one of different phases or to lower dimensional interfaces between the distinct phases.
Mathematically this can be described by a family of phase indicator functions (characteristic functions of subsets).

In the simplest case of an open domain $\Omega\subset\R^n$ and two phases it is sufficient to consider one phase indicator function $u:\Omega\to\{0,1\}$.
The relevant energy is in many cases proportional to the surface area of the phase interface,
which can be formulated as a perimeter functional for the phase indicator function,
\begin{equation*}
  \calP(u) = 
  \begin{cases}
    \int_\Omega |\nabla u|\quad&\text{if $u$ is of bounded variation,}\\
    +\infty &\text{ else.}
  \end{cases}
  \label{1eq:Per}
\end{equation*}

In the diffuse interface (or phase field) approach one allows for mixtures between phases and considers concentration fields of the different phases, which are typically taken as smooth functions on the considered domain.

A typical example of a diffuse interface energy is the Modica--Mortola or Van der Waals--Cahn-Hilliard energy, given by
\begin{equation}
  \calP_\eps(u) := \int_{\Omega} \Big(\eps |\nabla u|^2 +
  \frac{1}{\eps}W(u)\Big)\,d\L^n\,,
  \label{1eq:CH}
\end{equation}
where $W$ is a suitable nonnegative double-well potential with $\{W=0\}=\{0,1\}$ and $u$ is a smooth function on $\Omega\subseteq\R^n$.
To achieve low energy values, the function $u$ has to be close to the wells of the potential except for thin transition layers with thickness of order $\eps$.

One key question is  whether a given diffuse interface model reduces to a sharp interface model in the \emph{sharp interface limit} $\eps\to 0$.
In the case of the energy \eqref{1eq:CH} this has been made rigorous in the celebrated result by Modica and Mortola \cite{MoMo77,Modi87}:
the functionals $\calP_\eps$ converge in the sense of $\Gamma$-convergence to the perimeter functional,
\begin{equation*}
  \calP_\eps \to 2k \calP,\quad k = \int_{0}^1 \sqrt{W}.
  \label{eq:c0}
\end{equation*}

\medskip

In the present contribution we consider phase separation processes on generalized hypersurfaces, given by varifolds and currents.
We are interested in applications, where both the location of the generalized hypersurfaces and the separation into phases represent degrees of freedom.
An instance of this situation arises in the modeling of biomembranes, where the shape and internal organization of lipids and other components can adapt to the environment, see for example \cite{DJBS21}. The dynamic structural changes play a crucial role in numerous functions. Of course, dependence of some  energy functional on the location of the membrane and its separation into phases at the same time could occur in  a variety of other applications from different fields. 

\medskip

Variational models for multiphase membranes therefore depend on both the shape of the membrane and the internal composition.
Typical ingredients are a bending energy with phase-dependent parameters and a line tension energy between phases.
In the simplest situation of two phases, extensions of the classical one-phase shape energies of Canham and Helfrich \cite{Canh70,Helf73} have been proposed in the form of a Jülicher-Lipowsky energy \cite{JuLi96}
\begin{equation}
  \mathcal{E}(S_1,S_2) = \sum_{j=1,2}\int_{S_j} \Big(k_1^j (H-H_0^{j})^2+ k_2^j K\Big)\,d\Ha^{2} + \sigma\int_{\Gamma} 1\,d\Ha^1.
  \label{eq:E-twophase}
\end{equation}
Here the biomembrane is represented by a closed surface $S\subset\R^3$ that is decomposed into a disjoint union of open subsets $S_1,S_2$ of $S$ representing the phases, and their common boundary $\Gamma$.
The first integral represents a bending energy that involves  in general phase-dependent bending constants $k_1^j,k_2^j$ and the spontaneous curvature $H_0^j$, $j=1,2$.
The second integral in \eqref{eq:E-twophase} describes a phase separation (line tension) energy.
A simple prototype of the bending contribution is the Willmore energy, that is obtained in the case $H_0^j=0=k_2^j$, $j=1,2$.

Corresponding diffuse interface descriptions consider a smooth field $u_\eps$ that describes the phase decomposition  on the hypersurface $S$, and replace the line tension energy $\int_{\Gamma} 1\,d\Ha^1$ by a Modica--Mortola type approximation, which leads to energies of the form
\begin{equation*}
  \calE(S_\eps,u_\eps) = \int_{S_\eps} \Big(k_1(u_\eps) (H-H_0(u_\eps))^2+ k_2(u_\eps) K\Big)\,d\Ha^2 + \sigma \int_{S_\eps} \Big(\eps |\nabla u_\eps|^2 +
  \frac{1}{\eps}W(u_\eps)\Big)\,d\Ha^2.
\end{equation*}
In a regular setting, $S_\eps$ would be assumed to be a hypersurface of class $C^2$ and $u_\eps$ to be a smooth function on $S_\eps$, with $\nabla u_\eps$ denoting the tangential gradient on $S_\eps$.

A rigorous mathematical understanding of such energies is rather challenging and very little seems to be known in a general situation.
Reductions to rotational symmetry have been studied in \cite{ChMV13,Helm13,Helm15}.
The only variational analysis in the general case seems to be the recent work of Brazda et.~al.~\cite{BrLS20}.

One of the challenges is that bounds on curvature energies alone do not induce good compactness properties in classes of smooth surfaces.
Therefore it is necessary to consider generalized concepts of surfaces $S_\eps$, such as (oriented) integer rectifiable varifolds with a weak second fundamental form as introduced by Hutchinson \cite{hutchinson1986second}.
Such concepts have been used rather successfully to study Willmore or Canham--Helfrich type functionals.

In the case of phase separated membranes an additional challenge is to describe the decomposition into distinct phases and phase interfaces. 
In \cite{BrLS20} phases are characterized by oriented curvature varifolds with boundary, and suitable conditions are posed that guarantee an appropriate global structure.
Below we will present approaches that in contrast describe phases by smooth phase fields or generalized indicator functions.
This may offer a more concise description and embeds diffuse and sharp interface formulations in a common framework. 

\medskip

Another challenge is that the question of suitable concepts of Sobolev and $BV$ spaces on varifolds or currents is non-trivial. 
In the case of Sobolev spaces on varifolds, it seems that it is not possible to guarantee all the `good' properties that are present in the case of these spaces over open domains in $\R^n$, see the discussion in \cite{menne2016sobolev}.
Best suited for our purposes are $BV$ functions on currents as introduced by Anzellotti, Delladio and Scianna \cite{anzellotti1996bv} (see also \cite{Ossa97}) that are defined in terms of a generalized graph of a function over a current, see below for a precise definition.
The usefulness of this definition has been demonstrated in  \cite{anzellotti1996bv} by  providing  compactness and closure properties of the space of generalized $BV$ functions, as well as a coarea formula. 
The caveat of the compactness and closure statements is that they require additional (and rather restrictive) assumptions on the convergences of the underlying currents.
These are in particular expressed by a suitable strict convergence property of currents, similar to the strict convergence of $BV$ functions.
Our definition of Sobolev spaces on currents is characterized as an appropriate closure of smooth test functions in the ambient space and is justified using appropriate partial integration formulas on currents provided in \cite{anzellotti1996bv}.

\medskip

For  alternative formulations of Sobolev and  $BV$ functions with respect to measures see \cite{bouchitte2001convergence,BeBF99}.

\medskip

The main goal of the present paper is the derivation of rigorous sharp interface limits for diffuse phase separation energies of Modica--Mortola type for phase fields on varying generalized surfaces.
We consider a generalization of the energy \eqref{1eq:CH} for pairs of  currents and phase fields.
The currents are assumed to represent boundaries of finite perimeter sets in $\R^n$, and the phase fields are taken from the space of generalized $H^{1,p}$-functions with respect to the current.
Using such generalizations leads to a suitable formulation of a Modica--Mortola energy
in the form
\begin{equation}
  I_\eps(u_\eps,S_\eps)=\int_{\R^n} \big(\e|\nabla_{\mu_\e} u_\e|^2+\e^{-1}W(u_\e)\big) \d\mu_\eps\,,
  \label{1eq:MM-sharp}
\end{equation}
where $S_\eps$ is a the current representing the boundary of a set of finite perimeter, with associated  generalized surface area measure $\mu_\eps$, see Definition \ref{3def:MM} for a precise definition.

The limit sharp interface energy is a function of pairs consisting of a current and a phase indicator function.
We again consider currents $S$ that arise from integration over the boundary of finite perimeter sets.
The phase indicator functions $u$ belong to the space $BV(S)$ in the sense of \cite{anzellotti1996bv}.
We show that we can associate a generalized jump set $J_u$ to $u$ that consists of a $\Ha^{n-2}$-rectifiable set.
This in particular allows to assign a generalized phase interface area to this set, of the form
\begin{equation}
  I(u,S)=\Ha^{n-2}(J_u)\,.
  \label{1eq:MM}
\end{equation}

Our first main result  is a compactness and lower bound estimate of the functionals $I_\eps$ that correspond to a compactness and lower bound statement in the spirit of  Gamma-convergence, see Theorem \ref{thm:momolb} below.
As we rely on a compactness result from \cite{anzellotti1996bv} we in particular need to impose a crucial strict convergence assumption for the approximating currents.

As a second main contribution we present an application of this result to a Jülicher--Lipowsky type two-phase biomembrane energy of the form \eqref{eq:E-twophase}, where  for simplicity we restrict ourselves to a bending energy of Willmore type.
Here the strict convergence property needed for the application of  
Theorem \ref{thm:momolb} is enforced by the assumption that the Willmore energy of the approximating varifolds is small, allowing to deduce a unit density property by the Li--Yau inequality \cite{li1982new}.

Under these  assumptions we are able to prove a compactness and lower bound result, see Theorem \ref{4thm:main} below.

\medskip
The paper is organized as follows. 
In the next section we present some notations and recall some relevant concepts of generalized surfaces and generalized  $BV$ and Sobolev functions.
The main result on a sharp interface limit of our generalized Modica--Mortola energy for varying surfaces is contained in Section \ref{sec:3}, the application to two-phase biomembrane energies is the content of the final Section \ref{sec:4}.

\section{Notation and auxiliary results}

For a Radon measure $\mu$ in $\R^n$ and $m\leq n$ we denote by $\theta_m(\cdot,\mu)$ the $m$-dimensional density of $\mu$, 
\[
\theta_m(x,\mu)=\lim_{r\to 0}\frac{\mu(B(x,r))}{\alpha(m) r^m}\,,
\]
(if the latter exists), with $\alpha(m)$ being the volume of the unit ball in $\R^m$, $\alpha(m)=\calL^m(B(0,1))$.

Convergence of Radon measures in $\R^n$ means weak* convergence in $C^0_c(\R^n)$ and is denoted by `$\wsto$'.
Analogously, convergence as currents and as (oriented) varifolds means weak* convergence in the respective dual spaces (see below) and is denoted by `$\wsto$'.

\subsection{Currents and $BV$ functions on currents}
We briefly recall the definition of (rectifiable) currents and in particular the notion of $BV$ functions on currents from \cite{anzellotti1996bv}.

We denote by $\Lambda_k(\R^N)$, $0\leq k\leq N$ and by $\Lambda^k(\R^N)$ the spaces of all $k$-vectors and $k$-covectors, respectively, in $\R^N$.
We call $v$ a \emph{simple $k$-vector} if $v$ can be written as $v=v_1\wedge \ldots\wedge v_k$. 
With $\Lambda(N,k)=\{\alpha=(\alpha_1,\dots,\alpha_k)\,:\, 1\leq \alpha_1<\dots<\alpha_k\leq N\}$, $(e_1,\dots,e_N)$ the standard orthonormal basis of $\R^N$ and $(\dx^1,\dots,\dx^N)$ the corresponding dual basis of $\Lambda^1(\R^N)$ we can represent any $v\in \Lambda_k(\R^N)$ and any $\omega\in \Lambda^k(\R^N)$ uniquely as
\begin{equation*}
  v=\sum_{\alpha\in \Lambda(N,k)} a^\alpha e_\alpha,\qquad
  \omega=\sum_{\alpha\in \Lambda(N,k)} a_\alpha  \dx^\alpha,
\end{equation*}
with $a^\alpha,a_\alpha\in\R$, $e_\alpha=e_{\alpha_1}\wedge\ldots\wedge e_{\alpha_k}$, $\dx^\alpha=\dx^{\alpha_1}\wedge\ldots \dx^{\alpha_k}$ for all $\alpha\in \Lambda(N,k)$.
This representation induces a scalar product and an induced norm $|\cdot|$ on $\Lambda_k(\R^N)$ and $\Lambda^k(\R^N)$. The canonical Hodge star isomorphism is denoted by $*:\Lambda_k(\R^N)\to\Lambda_{N-k}(\R^N)$.

\smallskip
For $U\subset \R^N$ open and $k \in \{0,\dots,N\}$ we denote by $\calD^k(U)$ the space of all infinitely differentiable $k$-differential forms $U\to\Lambda^k(\R^N)$ with compact support in $U$, equipped with usual topology of distributions.

The space $\calD_k(U)$ of \emph{$k$-currents on $U$} is the dual of $\calD^k(U)$. 
We denote by $\partial T\in \calD_{k-1}(U)$ the \emph{boundary} of $T \in \calD_k(U)$, defined by
\begin{equation*}
  \angles{\partial T,\omega} = \angles{T,d\omega}\quad\text{ for all }\omega\in \calD^{k-1}(U).
\end{equation*}

We say a $k$-current $T$ on $U\subset \R^N$ is representable by integration if 
\[
  \mass[U]{T}:=\sup\{\angles{T,\varphi} \,:\,\varphi\in \calD^k(U),\|\varphi\|_{L^\infty}\leq 1\}<\infty\,.
\]
In that case by the Riesz representation theorem there exists a Radon measure $\|T\|$ on $U$ and a $\|T\|$-measurable function $\vec T:U\to \Lambda_k(\R^N)$ satisfying $|\vec T|=1$ a.e.~on $U$ such that
\begin{equation}\label{finmass-curr}
  \angles{T,\omega} :=\int_U \angles{\omega,\vec T}\,d\|T\|
  \quad\text{ for all }\omega\in\calD^k(U).
\end{equation}
We call $\|T\|$ the \emph{mass measure} of $T$ and $\mass{T}=\mass[U]{T}$ the total mass (in $U$). We write $\supp T\equiv\supp \|T\|$, where the latter is the largest (closed) subset of $U$ for which every open neighborhood of every point in the set has positive $\|T\|$-measure.

\smallskip
Given a $k$-rectifiable set $M\subset\R^N$ for $\calH^k$-almost any $p\in M$ there is a well-defined measure-theoretic tangent space $T_p M$. 
We say that a map $\tau: M\to \Lambda_k(\R^N)$ is an \emph{orientation} on $M$ if such a map is $\calH^k$-measurable and $\tau(p)$ is a unit simple $k$-vector on $\R^N$ that spans $T_pM$ for $\calH^k$-almost any point $p\in M$.  

Let $U\subset \R^N$ be open, $M\subset U$ be $k$-rectifiable, $\tau$ be an orientation on $M$, and $\rho \colon M \to \R^+$ be locally $\calH^k$-summable function. 
Then, we define a current $T=\curr{M,\rho,\tau} \in \calD_k(U)$ by
\begin{equation}\label{rect-curr}
  \langle T,\omega\rangle:=\int_M \langle \omega,\tau\rangle\,\rho\,d\calH^k
  \quad\text{ for all }\omega\in\calD^k(U).
\end{equation}
The set $\sfR_k(U)$ of currents $T\in \calD_k(U)$ which can be written in the form $T=\curr{M,\tau,\rho}$ as above are called \emph{rectifiable currents}, the function $\rho$ is then called the \emph{multiplicity} of $T$. 
 
If in addition $\rho$ is integer-valued we call $T$ \emph{integer-rectifiable} and write $T\in \sfI_k(U)$.
A current $T\in \sfI_k(U)$ with $\partial T\in \sfI_{k-1}(U)$ is called \emph{integral}.

In the context of graphs over sets in $\R^n$ it is often useful to write $\R^{n+1}=\R^n_x\times\R_y$, to denote the standard basis in $\R^{n+1}$ by $(e_1,\dots,e_n,e_y)$, and to use coordinates $(x,y)=(x_1,\dots,x_n,y)$ relative to this basis.
The associated covector basis in $\Lambda^1(\R^{n+1})$ is written as $dx^1,\dots,dx^n,dy$.

The \emph{stratification} of a $k$-vector $\xi\in \Lambda_k(\R^n_x\times \R_y)$ is given by the unique decomposition
\begin{align}
	\xi=\xi_0+\xi_1,\quad \xi_0\,\in\, \Lambda_k(\R^n_x),\quad \xi_1\in \Lambda_{k-1}(\R^n_x)\wedge \Lambda_1(\R_y),
  \label{eq:def-strati}
\end{align}
that is
\begin{align*}
	\xi_0 &= \sum_{\alpha\in \Lambda(n,k)} 
	\angles{\dx^\alpha,\xi} e_\alpha,\qquad
	\xi_1 = \sum_{\beta\in \Lambda(n,k-1)} 
	\angles{\dx^\beta\wedge \dy,\xi} e_\beta\wedge e_y.
\end{align*}
The corresponding stratification of a current $T\in \Lambda_k(\R^n_x\times \R_y)$ is given by 
\begin{align*}
	\angles[\Big]{T_0,\sum_{\alpha\in \Lambda(n,k)} a_\alpha \dx^\alpha+ \sum_{\beta\in \Lambda(n,k-1)}a_\beta \dx^\beta\wedge \dy} &= \angles[\Big]{T,\sum_{\alpha\in \Lambda(n,k)} a_\alpha \dx^\alpha},\\
	\angles[\Big]{T_1,\sum_{\alpha\in \Lambda(n,k)} a_\alpha \dx^\alpha+ \sum_{\beta\in \Lambda(n,k-1)}a_\beta \dx^\beta\wedge \dy} &= \angles[\Big]{T,\sum_{\beta\in \Lambda(n,k-1)}a_\beta \dx^\beta\wedge \dy},
\end{align*}
for $a_\alpha,a_\beta\in C^\infty_c(\R^n_x\times \R_y)$.
In the case that $T=\curr{M,\xi,\rho}$ is rectifiable we obtain for $j=0,1$
\begin{equation*}
  \angles{T_j,\omega} = \int_M \angles{\omega,\xi_j}\rho \,d\calH^k 
  \quad\text{ for all }\omega\in\calD^k(\R^n_x\times \R_y).
\end{equation*}

\smallskip
Let $M\subset \R^n_x\subset \R^{n+1}$ be a $k$-rectifiable set, and let $p:\R^{n+1}\to\R^n_x$ denote the projection on $\R^n_x$. 

For the rectifiable $k$-current $S=\curr{M,\tau,\rho}$ and a function $u:M\to \R$ we consider the set between the graph of $u$ and $\{y=0\}$ in $\R^{n+1}$,
\begin{equation*}
  E_{u,S} = \{(x,y)\in M\times\R_y\,:\, 0<y<u(x)\text{ if }u(x)>0,\,u(x)<y<0\text{ if }u(x)<0\}
\end{equation*}
and define for $(x,y)\in E_{u,S}$ an induced orientation and induced multiplicity by
\begin{align*}
  \alpha(x,y)&=
  \begin{cases}
    e_y\wedge\tau(x) & \text{ if } y>0\\
    -e_y\wedge\tau(x) & \text{ if } y<0
  \end{cases}
  \\
  \theta(x,y)&=\rho(x).
\end{align*}
We set 
\[
\Sigma_{u,S}:=[\![E_{u,S},\alpha,\theta]\!]
\]
and obtain the \emph{generalized graph of $u$ over $S$} as
\begin{equation}\label{eq:1}
  T_{u,S} = -\partial \Sigma_{u,S} +S\otimes \delta_0\,,
\end{equation}
where $S\otimes \delta_0$ is defined by $\langle S\otimes\delta_0,\varphi\rangle=\langle S,\varphi(\cdot,0)\rangle$.

\begin{definition}
  Let a rectifiable $k$-current $S=\curr{M,\tau,\rho}$, $M\subset\R^n_x$ and a function $u:M\to\R$ be given.
  Then we say that $u$ is a function of bounded variation over $S$ and write $u\in BV(S)$ if the total mass of $T_{u,S}$ is bounded, $\mass{T_{u,S}}<\infty$.
\end{definition}
\begin{remark}
From the definition, we straightforwardly get 
\begin{equation}\label{eq:7}
  \begin{split}
    T_{u,S}\left(\varphi_\alpha\dx^\alpha\right)
    &=(T_{u,S})_0\left(\varphi_\alpha\dx^\alpha\right)\\
    &= -\int_M\langle \dx^\alpha,\tau\rangle \varphi_\alpha(x,u(x))\rho(x)\d \H^k(x)\\
    T_{u,S}\left(\varphi_\beta\dx^\beta\wedge \d y\right)&=(T_{u,S})_1\left(\varphi_\beta\dx^\beta\wedge \d y\right)\\
&=(-\partial\Sigma_{u,S}\ecke\d y)(\varphi_\beta\d x^\beta)\\
    &=-\int_M\int_{0}^{u(x)}\sum_{i=1}^n \angles{\dx_i\wedge \dx^\beta,\tau} \frac{\partial \varphi_\beta(x,y)}{\partial x_i}\d y \rho(x)\d\H^k(x)
  \end{split}
\end{equation}
 for $u\in BV(S)$, $\alpha \in \Lambda(n,k)$ and
    $\beta\in\Lambda(n,k-1)$.
\end{remark}

\begin{remark}
\label{rem:TuSrect}
If $S=\curr{M,\tau,\rho}$ is in $\sfI_k(\R_x^n)$ and $u\in BV(S)$ then the boundary rectifiability theorem \cite[Theorem 30.3]{Simo83} implies that $\partial\Sigma_{u,S}$ is integer rectifiable and hence $T_{u,S}\in \sfI_k(\R_x^n\times\R_y)$.
Therefore there exists a $k$-rectifiable set $\calR\subset \R^{n+1}$, a multiplicity function $\theta:\calR\to\N$ and an orientation $\xi:\calR\to \Lambda_k(\R_x^n\times\R_y)$ such that 
\begin{equation*}
  T_{u,S} = \curr{\calR,\xi,\theta}.
\end{equation*}
\end{remark}

To induce good closure and compactness properties of the generalized graphs in \cite{anzellotti1996bv} the following \emph{strict convergence} property for currents was was introduced.
We denote by $q:\R^{n+1}\to\R$ the projection on the vertical component, $q(x,y)=y$.
\begin{definition}[Strict convergence]
  Let $T_j$ be a sequence of $n$-dimensional integer multiplicity rectifiable currents in $\R^{n+1}$ such that
  \begin{itemize}
  \item[(i)]$T_j\wsto T$\,,
  \item[(ii)] $\sup_j(\mass{T_{j,0}\ecke q}+\mass{T_j})<\infty$\,,
  \item[(iii)]  $\lim_{j\to\infty} \mass{p_\# T_j}=\mass{p_\#T}$\,.
  \end{itemize}
Then we say that $T_j$ converges strictly to $T$, $T_j\stackrel{c^*}{\wto}T$.
\end{definition}

\begin{remark}
  In the case that $T_j=T_{u_j,S_j}$, $T=T_{u,S}$ are generalized graphs with $S_j=\curr{M_j,\tau_j,\rho_j}$ and $S=\curr{M,\tau,\rho}$ being $k$-rectifiable currents we obtain
  \begin{equation*}
    \mass{T_{j,0}\ecke q} = \int_{M_j} |u_j|\rho_j\,d\Ha^k
    =\|u_j\|_{L^1(\|S_j\|)}\,.
  \end{equation*}
  
  Moreover, in this case the convergence $T_j\wsto T$ implies
  \begin{equation*}
    S = p_{\#}{T_{u,S}} = \lim_{j\to\infty} p_{\#}{T_{u_j,S_j}} = \lim_{j\to\infty}S_j,
  \end{equation*}
  and the convergence in the third item is equivalent to $\mass{S_j}\to\mass{S}$, i.e.
  \begin{equation*}
    \int_{M_j}\rho_j\,d\Ha^k \,\to\, \int_M \rho\,d\Ha^k.
  \end{equation*}
\end{remark}

\subsection{Sobolev functions with respect to currents}
We now define Sobolev functions with respect to currents, following \cite{bouchitte2001convergence}.

Let $S=\curr{ M,\tau,\rho}$ be as above, and $\mu=\|S\|$. 
The projection onto the tangent space of $M$ at $x$ is defined for $\H^k\ecke M$ almost every $x\in M$; we will denote it by $P(x)$. 
For $u\in C^\infty_c(\R^n)$ we set 
\[
\nabla_{\mu} u(x):=P(x)\nabla u(x)\,.
\]
For any $p\in [1,\infty)$ this defines an operator $\nabla_{\mu}:C^\infty_c(\R^n)\subset L^p_\mu(\R^n)\to L^p_\mu(\R^n;\R^n)$.
By Lemma \ref{lem:gradient_uniqueness_current} below this operator is closable and we define $H^{1,p}(S)$ as the domain of the unique closed extension of $\nabla_{\mu}$.
This domain coincides with the closure of $C^\infty_c(\R^n)$ with respect to the norm 
\[
\|u\|_{H^{1,p}(S)}=\|u\|_{L^p_\mu(\R^n)}+\|\nabla_\mu u\|_{L^p_\mu(\R^n;\R^n)}\,.
\]
We use the same notation as in the smooth case also for the extension $\nabla_{\mu}:H^{1,p}(S)\to L^p_\mu(\R^n;\R^n)$.

\medskip

The closability of $\nabla_{\mu}:C^\infty_c(\R^n)\subset L^p_\mu(\R^n)\to L^p_\mu(\R^n;\R^n)$ is guaranteed by the following lemma.
For simplicity we will assume that $\partial S=0$, since this is the relevant case in the following sections. 
Also, we will use the measure representation of $BV$ functions from \cite[Corollary 2.10(ii)]{anzellotti1996bv}, which states that whenever $S=\curr{ M,\tau,\rho}$ and $u\in BV(S)$, then there exists a $\Lambda_{n-k+1}(\R^n)$ valued measure $R(S,u)$ whose components $R^\gamma(S,u)$, $\gamma\in \Lambda(n,n-k+1)$, are given by 
\[
\langle \varphi,R^\gamma(S,u)\rangle=\int_M u (\nu\wedge\nabla \varphi )^\gamma \rho\d\H^k \quad\text{ for all } \varphi\in C^1_c(\R^n)\,.
\]

\begin{lemma}
\label{lem:gradient_uniqueness_current}
Assume that $S$ is as above with $\partial S=0$. 
Then for any $(v_j)_{j\in \N}\subset C^\infty_c(\R^n)$, $w\in L^p_{\mu}(\R^n;\R^n)$ with
\begin{equation*}
  v_j \to 0 \,\text{ in }L^p_{\mu}(\R^n)\quad\textrm{ and }\quad v_j\to w\,\text{ in }L^p_{\mu}(\R^n;\R^n)\,,
\end{equation*}
we have $w=0$.
\end{lemma}

\begin{proof}
Let $\varphi\in C_c^\infty(\R^n)$. 
This implies that $R(S,\varphi)$ is in $C^\infty_c(\R^n;\Lambda_{n-k+1}(\R^n))$, and we obtain
\[
0=\lim_{j\to\infty} \int_M R^\gamma(S,\varphi)v_j\rho\d\H^k=-\lim_{j\to\infty}\int_M R^\gamma(S,v_j)\varphi\rho\d\H^k\, 
\]
for $\gamma\in\Lambda(n,n-k+1)$.
Using the fact $R(S,v_j)=\nu\wedge\nabla_\mu v_j$ (see \cite[Proposition 1.2]{anzellotti1996bv}) 
we deduce that 
\begin{equation*}
  0=\lim_{j\to\infty}\int_M R^\gamma(S,v_j)\varphi\rho\d\H^k
  =\int (\nu\wedge\nabla_\mu v_j)^\gamma\varphi\rho\d\H^k
  =\int (\nu\wedge w)^\gamma\varphi\rho\d\H^k\,.
\end{equation*}
By the fact that $\nabla_\mu v_j$ is orthogonal to $\nu$, the same holds for the vector field $w$, and hence $w=0$ follows since $\varphi$ can be chosen arbitrarily.
\end{proof}

\subsection{Oriented varifolds}
We introduce oriented varifolds as in \cite{hutchinson1986second}. 
Let $G^o(n,m)$ denote the set of $m$-dimensional oriented subspaces of $\R^n$. 
This set may be identified with the simple unit $m$-vectors in $\Lambda_m(\R^n)$. 
An oriented $m$-varifold is an element of $\mathcal M(\R^n\times G^o(n,m))$. 

If $M$ is $m$-rectifiable with orientation $\tau$, and if $\theta_\pm:M\to\R^+_0$ are locally $\Ha^m$-summable multiplicity functions with $\theta_++\theta_->0$, we write $\underline{v}(M,\tau,\theta_\pm)$ for the associated rectifiable oriented varifold
\[
\underline{v}(M,\tau,\theta_\pm)(\varphi)=\int_{M\times G^o(n,m)}\big(\theta_+(x)\varphi(x,\tau(x))+\theta_-(x)\varphi(x,-\tau(x))\big)\d \H^m(x)\,.
\]
We then can pass to a representation $\ovar{M,\tilde\tau,\tilde\theta_\pm}$ such that $\tilde\theta_+\geq\tilde\theta_-$.
In the following we always assume this additional property.

If the multiplicity functions $\theta_\pm$ are $\N_0$-valued, then we say that $\underline{v}(M,\tau,\theta_\pm)$ is an integral oriented $m$-varifold. 
The class of $m$-dimensional oriented (resp.~rectifiable oriented, resp.~integral oriented) varifolds is denoted by $\sfV^o_m(\R^n)$ (resp.~$\sfRV^o_m(\R^n)$, resp.~$\sfIV^o_m(\R^n)$). 
To $V\in \sfV^o_m(\R^n)$, we associate the $m$-dimensional current
\[
\underline{c}(V)(\varphi)=\int_{\R^n\times G^o(n,m)}\langle \varphi(x),\xi\rangle \d V(x,\xi)
\]
and observe that in the case of $V\in\sfRV^0_m(\R^n)$, $V=\underline{v}(M,\tau,\theta_\pm)$
\begin{equation*}
  \underline{c}(V)(\varphi)=\int_M\langle \tau(x),\varphi(x)\rangle (\theta_+-\theta_-)(x)\d \Ha^m(x).
\end{equation*}
We remark that convergence as oriented varifolds implies convergence of the associated currents,
\begin{equation}
  V_j \wsto V  \quad\implies\quad \underline{c}(V_j) \wsto \underline{c}(V)\,.
  \label{2eq:orvcurr}
\end{equation}





\subsection{Measure-function pairs}
We use the definitions of measure-function pairs from \cite{moser2001generalization}, which in turn are based on \cite{hutchinson1986second}. Let $\Omega\subset \R^n$.
If $\mu\in \mathcal{M}(\R^n)$ and $f\in L^1_{\mathrm{loc},\mu}(\Omega;\R^m)$, then we say that $(\mu,f)$ is a measure-function pair over $\Omega$ with values in $\R^m$. 
\begin{definition}
Let $\{(\mu_k,f_k):k\in\N\}$ and $(\mu,f)$ be measure-function pairs over $\Omega$ with values in $\R^m$, and $1\leq p<\infty$. 
\begin{itemize}
 \item[(i)] We say that $(\mu_k,f_k)$ converges weakly in $L^p$ to $(\mu,f)$ and write 
\[
(\mu_k,f_k)\wto (\mu,f) \quad\text{ in } L^p
\]
if $\mu_k\wsto \mu$ in $\mathcal M(\Omega)$, $\mu_k\ecke f_k\wsto \mu\ecke f$ in $\mathcal M(\Omega;\R^m)$, and $\|f_k\|_{L^p_{\mu_k}(\Omega;\R^m)}$ is uniformly bounded.
\item[(ii)] We say that $(\mu_k,f_k)$ converges strongly  in $L^p$ to $(\mu,f)$ and write 
\[
(\mu_k,f_k)\to (\mu,f) \quad\text{ in } L^p
\]  
if for all $\varphi\in C_c^0(\Omega\times\R^m)$, 
\[
\lim_{k\to\infty}\int_\Omega \varphi(x, f_k(x))\d\mu_k(x)=\int_\Omega \varphi(x, f(x))\d\mu(x)\,,
\]
and
\[
\lim_{j\to \infty}\int_{S_{kj}}|f_k|^p\d\mu_k=0 \quad\text{ uniformly in } k\,,
\]
where $S_{kj}=\{x\in\Omega:|x|\geq j \text{ or } |f_k(x)|\geq j\}$.
 \end{itemize}
\end{definition}

\subsection{Convergence of BV functions over varying currents}
\label{sec:conv-bv-funct}
Under suitable assumptions the strict convergence of currents associated to generalized graphs induces a measure-function pair convergence that will be very useful in the following.
\begin{proposition}\label{2pro:mfp-convergence}
Consider a sequence $(S_j)_j$ of $(n-1)$-rectifiable currents in $\R^n$.
Moreover, let a sequence $(u_j)_j$ be given with $u_j\in BV(S_j)$, and denote the associated generalized graphs by $T_j=T_{u_j,S_j}$.

Assume that 
\begin{equation}
  \sup_j \|u_j\|_{L^p_{\|S_j\|}}<\infty \quad\text{ for some }1<p<\infty\,, 
  \label{2eq:u_bd}
\end{equation}
and that for some $(n-1)$-rectifiable current $S$ and some $u\in BV(S)$
\begin{align}
  T_j  \wsto T=T_{u,S}\,, \label{2eq:ass-gcurr}\\
  \mass{S_j} \to \mass{S}\,. \label{2eq:Resh}
\end{align}
Then the strong measure-function pair convergence
\begin{equation}
  \big(\|S_j\|,u_j\big)  \to \big(\|S\|,u\big)
  \label{2eq:mfp-conv}
\end{equation}
holds in any $L^q$, $1\leq q<p$.
\end{proposition}

\begin{proof}
We first deduce from \eqref{2eq:u_bd} for some $\Lambda>0$ and any $j\in\N$, $1\leq q<p$, $R>0$
\begin{equation*}
  \Lambda^p \geq \int |u_j|^p\d\|S_j\| \geq R^p\|S_j\|(\{|u_j|>R\})
\end{equation*}
and
\begin{align*}
  \int_{\{|u_j|>R\}}|u_j|^q\d\|S_j\|
  &\leq \Lambda^q \Big(\|S_j\|(\{|u_j|>R\})\Big)^{1-\frac{q}{p}} \\
  &\leq \Lambda^p R^{-p+q}\,\to\, 0 \quad(R\to\infty).
\end{align*}
Similarly,
\begin{align*}
  \int_{\{|x|>R\}}|u_j|^q\d\|S_j\|
  &\leq \Lambda^q \sup_{j\in\N}\Big(\|S_j\|(\{|x|>R\})\Big)^{1-\frac{q}{p}} \\
  \,\to\,0 \quad (R\to\infty)
\end{align*}
by Prokhorov's Theorem \cite[Theorem 8.6.2]{bogachev2007measure} and \eqref{2eq:ass-gcurr}, \eqref{2eq:Resh}.
This verifies the second condition in the definition of strong measure-function pair convergence in $L^q$.

\medskip
Next, for any $\alpha\in \Lambda(n,n-1)$, any $\varphi_\alpha\in C^0_c(\R^n)$, and any $\psi\in C^0_c(\R^n\times\R)$ we let $\eta_\alpha(x,y)=\varphi_\alpha(x)\psi(x,y)$ and deduce from \eqref{finmass-curr} and \eqref{eq:7}
\begin{align*}
  &\int_{\R^n} \angles{\vec{S}(x),\dx^\alpha} \varphi_\alpha(x)\psi(x,u(x))\,\d\|S\|(x)\\
  &\qquad =\angles{S,\psi(\cdot,u)\varphi_\alpha\dx^\alpha}\\
  &\qquad =  -\angles{T_{u,S},\eta_\alpha\dx^\alpha} \\
  &\qquad = -\lim_{j\to\infty} \angles{T_{u_j,S_j},\eta_\alpha\dx^\alpha}\\
  &\qquad = \lim_{j\to\infty} \angles{S_j,\psi(\cdot,u_j)\varphi_\alpha\dx^\alpha}
  = \lim_{j\to\infty }\int_{\R^n}\angles{\vec{S_j},\dx^\alpha} \varphi_\alpha(x)\psi(x,u_j(x))\,\d\|S_j\|(x).
\end{align*}
Since $\alpha\in \Lambda(n,n-1)$, $\varphi_\alpha$ are arbitrary we arrive at
\begin{equation}
  \Big(\psi(\cdot,u_j)\vec{S_j}, \|S_j\|\Big)
   \weakto \Big(\psi(\cdot,u)\vec{S}, \|S\|\Big).
   \label{2eq:Moser}
\end{equation}
Next, by the Reshetnyak continuity theorem (see \cite{reshetnyak1968weak}, \cite[Theorem 1.3]{spector2011simple}) and \eqref{2eq:Resh} we have for any $\varphi\in C^0_c(\R^n\times\sphere^{n-1})$
\begin{align*}
  \lim_{j\to\infty}\int_{\R^n}\varphi(\cdot,\vec{S_j})\d\|S_j\| 
  &= \int_{\R^n}\varphi(\cdot,\vec{S})\d\|S\|,
\end{align*}
which implies the strong measure-function pair convergence
\begin{equation*}
  \big(\vec{S_j},\|S_j\|\big) \to \big(\vec{S},\|S\|\big)
\end{equation*}
in any $L^r$, $1\leq r<\infty$. 
Together with \eqref{2eq:Moser} we deduce from \cite[Proposition 3.2]{moser2001generalization} that \begin{align*}
  \lim_{j\to\infty}\Big(\psi(\cdot,u_j),\|S_j\|\Big)&=\lim_{j\to\infty}\Big(\psi(\cdot,u_j)\vec{S_j}\cdot\vec{S_j}, \|S_j\|\Big)\\
  &= \Big(\psi(\cdot,u)\vec{S}\cdot\vec{S}, \|S\|\Big)
  = \Big(\psi(\cdot,u), \|S\|\Big)
\end{align*}
in the sense of weak measure-function-pair convergence in any $L^r$, $1\leq r<\infty$.
Since $\psi$ was arbitrary \eqref{2eq:mfp-conv} follows.
\end{proof}

\section{Compactness and lower semicontinuity for a Modica-Mortola functional on varying surfaces}
\label{sec:3}

In this section we provide a lower bound estimate for a sequence of diffuse perimeter approximations on varying surfaces.
We first define a generalization of the Modica--Mortola functional to this situation.

Fix a nonnegative continuous double-well potential $W$ such that $\{W=0\}=\{0,1\}$ and such that for some $T,c>0$, $p\geq 2$
\begin{equation}
  c|t|^p \leq W(t)\leq \frac{1}{c}|t|^p \quad\text{ for all }|t|\geq T\,.
  \label{3eq:dw}
\end{equation}
Below it will be convenient to fix a first integral of $\sqrt{W}$,
\begin{equation}\label{eq:12}
  \dwp(r)=\int_{0}^r \sqrt{W(t)}\d t\quad\text{ for }r\in\R
\end{equation}
and to define the surface tension constant
\begin{equation}
  k=\int_{0}^{1}\sqrt{W(r)}\d r\,.
  \label{3eq:k}
\end{equation}

\begin{definition}[Generalized Modica--Mortola functional]
  \label{3def:MM}
For $\eps>0$, a rectifiable current $S_\e=\curr{M_\e,\tau_\e,\rho_\e}$ of finite mass in  $\R^n$ and a function $u_\eps\in H^{1,p}(S_\e)$, we define the Modica--Mortola-type functional
\begin{equation}
  I_\eps(u_\eps,S_\e)=\int_{\R^n} \big(\e|\nabla_{\mu_\e} u_\e|^2+\e^{-1}W(u_\e)\big) \d\mu_\eps\,,
  \label{3eq:MM}
\end{equation}
where $\mu_\e=\|S_\e\|$.
\end{definition}

In the following we consider a more restrictive setting where $S_\eps$ is the current associated to  the reduced boundary of a finite perimeter set.
In the sharp interface limit of phase fields we will obtain phase indicator functions with a generalized $BV$-regularity.

In Proposition \ref{pro:countval} below we will justify the following notion of jump sets for phase indicator functions.
\begin{definition}[Generalized jump set]\label{3def:Ju}
Let $S\in \sfI_m(\R^n)$, $S=\curr{M,\tau,\rho}$ and $u\in BV(S)$ be given with $u(x)\in \{0,b\}$ for almost all $x\in M$ and some $b>0$.
Then we write
\[
  J_u:= \supp \left(\partial\curr{u^{-1}(b),\tau,\rho}\right)
\]
and call this set the jump set of $u$.
\end{definition}

\begin{proposition}
  \label{pro:countval}
  Consider $S\in \sfI_{m}(\R^n)$, $S=\curr{M,\tau,\rho}$ and $u\in BV(S)$ with $u(x)\in Q$ for almost all $x\in M$, $Q\subset\R$ at most countable.
  Then we have that
  \begin{equation}
    \mass{T_{u,S}}=\mass{\partial \Sigma_{u,S}\ecke\d y}+\mass{S}\,.
    \label{3eq:countval}
  \end{equation}
  In the case $Q=\{0,b\}$, $b>0$ as in Definition \ref{3def:Ju} we obtain that the generalized jump set $J_u$ is $(m-1)$-rectifiable and that
  \begin{equation}
    \mass{T_{u,S}}-\mass{S}=b\H^{m-1}(J_u)\,.
    \label{3eq:countval2}
  \end{equation}
\end{proposition}
  
\begin{proof}
Writing $T_{u,S} = \curr{\calR,\xi,\theta}$ as in Remark \ref{rem:TuSrect}, we may set
\[
\calR_1=\{(x,y)\in\calR:\xi_0(x,y)\neq 0\}\,,\quad\calR_0=\calR\setminus\calR_1\,.
\]
Then by \cite[Lemma 2.7]{anzellotti1996bv} we may write
\[
\calR_1=\calR_1^0\cup \bigcup_{r\in Q} \calR_1^{(r)}\,,
\quad\Ha^m(\calR_1^0)=0,\, \calR_1^{(r)}\subset M\times\{r\} \text{ for }r\in Q\,,
\]
and hence $\xi_1=0$ on $\calR_1$. 
This translates into the following identities for the stratification of $T_{u,S}$,
\[
  \begin{split}
(T_{u,S})_0&=T_{u,S}\ecke\{\xi_1=0\}\\
(T_{u,S})_1&=T_{u,S}\ecke\{\xi_0=0\}\,.
\end{split}
\]
Hence $\mass{T_{u,S}}=\mass{(T_{u,S})_0}+\mass{(T_{u,S})_1}$ and \eqref{3eq:countval} follows from \eqref{eq:7}.

By the coarea formula for $BV(S)$ (see \cite[Proposition 2.13 (ii)]{anzellotti1996bv}) we have that
\begin{equation}\label{eq:2}
  \begin{split}
    \mass{\partial \Sigma_{u,S}\ecke\d y}
    &=\int_0^\infty \mass{\partial \left(S\ecke\{u>s\}\right)}\d s\\
    &= b\mass{\partial \curr{u^{-1}(b),\tau,\rho}}\,.
  \end{split}
\end{equation}
By the boundary rectifiability theorem \cite[Theorem 30.3]{Simo83}, 
the generalized jump set $J_u=\supp\partial \curr{u^{-1}(b),\tau,\rho}$ is $(m-1)$-rectifiable, and 
\begin{equation*}
  \mass{\partial \curr{u^{-1}(b),\tau,\rho}}=\H^{m-1}(J_u)\,.
\end{equation*}
Together with \eqref{3eq:countval}, \eqref{eq:2} this implies \eqref{3eq:countval2}.
\end{proof}
  
The main result in this section provides a lower bound estimate analogous to the $\liminf$ estimate in the classical Modica--Mortola Gamma-convergence statement.

\begin{theorem}[Lower bound estimate]
\label{thm:momolb}
Consider a double-well potential $W$ as above with \eqref{3eq:dw} for some $p\geq 2$.
Let a family $(E_\eps)_\eps$ of finite perimeter sets in $\R^n$ with associated perimeter currents $S_\e=\curr{\partial_* E_\e,*\nu_\e,1}$, where $\nu_\e:\partial_* E_\e\to\sphere^{n-1}$ denotes the inner unit normal of $E_\e$,
 $\mu_\e=\|S_\e\|$, and a sequence $(u_\eps)_\eps$ in $H^{1,p}(S_\e)$ be given.

Assume that for some set $E$ of finite perimeter $\chi_{E_\e}\to \chi_{E}$ strictly in $BV(\R^n)$, that is
\begin{equation}
  \chi_{E_\e}\to \chi_{E}\,\text{ in }L^1(\R^n),\quad
  \nabla \chi_{E_\e}\wsto \nabla \chi_{E},\quad\text{and } 
  \lim_{\e\to 0}\H^{n-1}(\partial_* E_\e)=\H^{n-1}(\partial_* E)\,.
  \label{3eq:BV-strict}
\end{equation}
Let $\nu=\nu_E:\partial_* E\to\sphere^{n-1}$ denote the inner unit normal of $E$, and set $S=[\![\partial_*E,*\nu_E,1]\!]$, $\mu=\H^{n-1}\ecke\partial_* E$.

Let us further assume that for some $\Lambda>0$
\begin{equation}
  I_\eps(u_\eps,S_\e) < \Lambda.
  \label{3eq:ass-bound}
\end{equation}
Then there exists $u\in BV(S)$ and a subsequence $\eps\to 0$ such that the following holds:
\begin{align}
  u(x)\in\{0,1\}\text{ for }\H^{n-1}\text{-almost all } x\in \partial_*E,
  \label{3eq:thm1-1}\\
  \big(\mu_\eps,u_\eps\big) \to \big(\mu,u\big)
  \quad\text{ as measure-function pairs in }L^q
  \label{3eq:thm1-2}
\end{align}
for any $1\leq q<p$.

Moreover, $\curr{\{u=1\}, *\nu,1}$ is an integral $(n-1)$-current and we have the lower bound estimate
\begin{equation}
  \liminf_{\e\to 0} I_\eps(u_\eps,S_\e) \geq 2k\H^{n-2}\left(J_u\right),
  \label{3eq:lsc}
\end{equation}
with the generalized jump set $J_u$ as in Definition \ref{3def:Ju}.
\end{theorem}

\begin{proof}
We first prove that we can assume $u_\eps\in\calD(\R^n)$ for all $\eps>0$.
In fact, by the definition of $H^{1,p}(S_\e)$ we can approximate $(u_\eps)_\eps$ by a family $(\tilde u_\eps)_\eps$ in $\calD(\R^n)$ such that \eqref{3eq:ass-bound} holds for $(\tilde u_\eps)_\eps$, such that
\begin{equation*}
  \liminf_{\e\to 0} I_\eps(\tilde u_\eps,S_\eps)
  =\liminf_{\e\to 0} I_\eps(u_\eps,S_\eps)\,,
\end{equation*}
and
\begin{equation*}
   \|u_\eps-\tilde u_\eps\|_{L^p_{\mu_\eps}(\R^n)}
   +\|\nabla_{\mu_\eps}u_\eps-\nabla_{\mu_\eps} \tilde u_\eps\|_{L^p_{\mu_\eps}(\R^n;\R^n)}
   \to 0 \quad (\eps\to 0)\,.
\end{equation*}
Let us assume we have proved $(\mu_\eps,\tilde u_\eps)\to (\mu, u)$ as measure-function pairs in $L^q$ for some $u$ as in the statement of the present theorem.
For any Lipschitz-continuous function $\psi\in C^0_c(\R^n\times\R)$ we then deduce
\begin{equation*}
  \Big|\int \big(\psi(\cdot,u_\eps)-\psi(\cdot,\tilde u_\eps)\big)\,d\mu_\eps\Big|
  \leq \|\psi\|_{C^{0,1}(\R^{n+1})} \|u_\eps-\tilde u_\eps\|_{L^p_{\mu_\eps}(\R^n)}\mu_\eps(\R^n)^{1-\frac{1}{p}}\,\to\, 0
\end{equation*}
with $\eps\to 0$.
An approximation argument yields for all $\psi\in C^0_c(\R^n\times\R)$
\begin{equation*}
  \lim_{\eps\to 0}\int  \psi(\cdot,u_\eps)\,d\mu_\eps
  =\lim_{\eps\to 0}\int  \psi(\cdot,\tilde u_\eps)\,d\mu_\eps
  =\int  \psi(\cdot,u)\,d\mu
\end{equation*}
and therefore the strong measure-function pair convergence of the original sequence $(\mu_\eps,u_\eps)$ to $(\mu,u)$.

Therefore it is sufficient to prove the Theorem for sequences $(u_\eps)_\eps$ in $\calD(\R^n)$, which we assume in the remainder of the proof.

\medskip
We now restrict ourselves to a subsequence $\eps\to 0$ (not relabeled) such that $I_\eps(u_\eps,S_\eps)$ converges to the $\liminf$ in \eqref{3eq:lsc}.
Moreover, we consider the modified phase fields $v_\eps:=\dwp\circ u_\eps$ and the generalized graphs $T_\eps:=T_{v_\eps,S_\e}$.

Since $u_\eps\in\calD(\R^n)$ we can apply the chain rule $\nabla_{\mu_\eps}v_\eps=\sqrt{W(u_\eps)}\nabla_{\mu_\eps}u_\eps$ and obtain the representation 
\begin{equation*}
  T_\eps = (\Phi_\e)_\# (S_\e),\quad
  \Phi_\e(x):=(x,v_\eps(x))\text{ for }x\in \partial_* E_\eps\,.
\end{equation*}

Therefore we may apply the usual Modica-Mortola trick to obtain
\begin{equation}\label{eq:17}
  \begin{split}
    \int_{\partial_* E_\e} |\nabla_{\mu_\e}v_\eps|\d\H^{n-1}&\leq 
    \int_{\partial_* E_\e} \sqrt{W(u_\e)}|\nabla_{\mu_\e} u_\e|\d\H^{n-1}\\
    &\leq \frac12 I_\eps(u_\eps,S_\eps)\leq \frac{1}{2}\Lambda\,.
  \end{split}   
\end{equation}
Furthermore we have
\begin{equation}\label{3eq:mass-bound}
  \begin{split}
    \mass{T_\eps} =\mass{(\Phi_\e)_\# (S_\e)}
    &=\int_{\partial_*E_\e}\sqrt{1+|\nabla_{\mu_\e}v_\eps|^2}\d\H^{n-1}\\
    &\leq \H^{n-1}(\partial_*E_\e)+ \int_{\partial_* E_\e} |\nabla_{\mu_\e}v_\eps|\d\H^{n-1}\\
    &\leq \H^{n-1}(\partial_*E_\e)+ \frac12 I_\eps(u_\eps,S_\eps)\leq C(\Lambda)
  \end{split}
\end{equation}
and
\begin{equation}
  \mass{T_{\eps,0}\ecke q} = \|v_\eps\|_{L^1_{\mu_\eps}}
  \leq C(\Lambda,W).
  \label{3eq:L1-bound}
\end{equation}
This implies that the currents $T_\eps$ are uniformly bounded in mass and that $(v_\eps)_\eps$ is bounded in the desired $L^1$ sense.

The strict $BV$-convergence \eqref{3eq:BV-strict} yields
\begin{equation}
  \lim_{\e\to 0} \mass{S_\e}=\lim_{\e\to 0}\H^{n-1}(\partial_* E_\e)=\H^{n-1}(\partial_* E)=\mass{S}\,.
  \label{3eq:Se-conv}
\end{equation}

Since $\partial T_\eps=0$, we may use the compactness theorem for integral currents \cite[Theorem 27.3]{Simo83}, to obtain a limit $T$ of $T_\eps$. 

By \eqref{3eq:mass-bound}, \eqref{3eq:L1-bound} and \eqref{3eq:Se-conv} we may use \cite[Theorem 4.3]{anzellotti1996bv} to conclude that there exists $v\in BV(S)$ such that
\[
  T=T_{v,S}\,.
\]

Furthermore, the assumptions on $W$ and the boundedness of $I_\eps$ induce that $(v_\eps)_\eps$ is bounded in $L^p_{\mu_\eps}(\R^n)$.
Together with the strict $BV$-convergence of $\chi_{E_\eps}$ and the convergence $T_\eps \wsto T_{v,S}$ we obtain by Proposition \ref{2pro:mfp-convergence} that we have the strong measure-function pair convergence
\begin{equation*}
  \big(\mu_\eps,v_\eps\big) \to \big(\mu,v\big)\quad\text{ in any }L^q,\,1\leq q<p.
  \label{3eq:mfp-v}
\end{equation*}
Since $\dwp$ is continuous and $u_\eps$ is uniformly bounded in $L^p_{\mu_\eps}(\R^n)$ by the boundedness of $I_\eps$ we deduce that \eqref{3eq:thm1-2} holds for $u:=\dwp^{-1}(v)$.

Using the measure-function pair convergence \eqref{3eq:thm1-2} and the energy bound \eqref{3eq:ass-bound} it follows that for any $\eta\in C^0_c(\R^n\times\R)$
\begin{equation*}
  \int_{\R^n} \eta(\cdot,u)W(u)\,\d\mu
  =\lim_{\eps\to 0} \int_{\R^n} \eta(\cdot,u_\eps)W(u_\eps)\,\d\mu_\eps =0\,.
\end{equation*}
Since $\eta$ was arbitrary this yields $W(u)=0$ $\mu$-almost everywhere.
Hence \eqref{3eq:thm1-1} holds, which implies that $v(x)\in\{0,k\}$ for $\H^{n-1}$ almost all $x\in\partial E_*$.

We then also have $u=\frac{1}{k}v\in BV(S)$.

\medskip

We are now in a position to prove the $\liminf$-inequality \eqref{3eq:lsc}. 
By \eqref{3eq:countval2} in Proposition \ref{pro:countval} we deduce
\begin{equation*}
  k\H^{n-2}(J_u)=\H^{n-2}(J_v)=\mass{T_{v,S}}-\mass{S}\,.
\end{equation*}
Noting  that mass is weakly lower semicontinuous under convergence as currents, and by \eqref{3eq:Se-conv} we get 
\[
\mass{T_{v,S}}-\mass{S}\leq \liminf_{\e\to 0}\mass{T_\eps}-\mass{S_\e}\,.
\]
By \eqref{3eq:mass-bound}, 
\[
  \begin{split}
\liminf_{\e\to 0}\mass{T_\eps}-\mass{S_\e}
&\leq \liminf_{\eps\to 0}\Big(\big(\H^{n-1}(\partial_* E_\eps)+\frac{1}{2}I_\eps(u_\eps,S_\eps)\big) - \H^{n-1}(\partial_* E_\eps)\Big)\\
&=\liminf_{\eps\to 0} \frac{1}{2}I_\eps(u_\eps,S_\eps)\,\,.
\end{split}
\]
By the above chain of equalities and inequalities, we have proved  \eqref{3eq:lsc}.
\end{proof}

\begin{remark}
For a corresponding Gamma-convergence result one would need to complement Theorem \ref{thm:momolb} by an upper bound estimate.
It is rather straightforward to prove such a statement in a more regular setting:
Assume that $M\subset \R^n$ is a smooth oriented $n-1$-dimensional submanifold,  $S=[\![M,*\nu,1]\!]$, $u\in BV(S)$ with $u\in \{0,1\}$ $\H^{n-1}\ecke M$ almost everywhere, $J_u$ a smooth curve on $M$. Then there exists a sequence $(u_\e)_{\e}$ in $H^{1,p}(S)$ such that $T_{u_\e,S}\wsto T_{u,S}$ and 
\[
\limsup_{\e\to 0} I_\e(u_\e,S)\leq 2k\H^{n-2}(J_u)\,.
\]  
Indeed, one can adapt the proof from \cite[Proposition 2]{Modi87} rather straightforwardly to obtain the above statement.

\medskip
The proof of the upper bound becomes non-trivial as soon as one allows for non-smooth surfaces $M=\supp S$. 
We do not discuss this question here and leave it open for future research.
\end{remark}

\begin{remark}[Mass constraint]
We remark that that one can  also include a constraint on the mass of $u_\eps$, $u$ and still obtain the lower semicontinuity in the restricted classes.
In fact, this follows since the measure-function pair convergence \eqref{3eq:thm1-2} implies
\begin{equation*}
  \int_{\R^n} u\,\d\mu = \lim_{\eps\to 0} \int_{\R^n} u_\eps\,\d\mu_\eps\,.
\end{equation*}
\end{remark}

Concerning the Modica--Mortola type functional for fixed $\eps>0$ we will now show that the existence of a minimizer can be obtained by following the same strategy as in the proof of Theorem \ref{thm:momolb}.
Without loss of generality we may assume $\eps=1$.

\begin{proposition}[Minimization of the Modica--Mortola type functional under a mass constraint]\label{3pro:MinMoMo}
Consider a double-well potential $W$ with $p$-growth at infinity, $p\geq 2$, a finite perimeter set $E\subset\R^n$ with associated perimeter current $S=\curr{\partial_*E,*\nu,1}$, where $\nu$ is the unit inner normal to $E$,  and define $I: H^{1,p}(S)\to [0,\infty)$ by
\begin{equation}
  I(u)=\int_{\R^n} \big(|\nabla_{\mu} u|^2+W(u)\big) \d\mu\,,
  \label{3eq:MMFix}
\end{equation}
where  $\mu=\H^{n-1}\ecke\partial_* E$.
Then for any $m>0$ there exists a minimizer of $I$ in the class $\big\{u\in H^{1,p}(S)\,:\, \int_{\R^n} u\,d\mu=m\big\}$.
\end{proposition}

\begin{proof}
We use the direct method of the Calculus of Variations and consider a minimal sequence $(u_k)_k$ in $\{u\in H^{1,p}(S)\,:\, \int_{\R^n} u\,d\mu=m\}$.
As in the proof of Theorem \ref{thm:momolb} we may assume that $u_k\in\calD(\R^n)$ for all $k\in\N$.
Since the constant function with value $\frac{m}{\mu(\R^n)}$ has finite energy, we may assume
\begin{equation}
  I(u_k)\leq \Lambda\quad\text{ for all }k\in\N
  \label{eq:BoundUk}
\end{equation}
for some $\Lambda>0$.
The assumption on $W$ implies that $(u_k)_k$ is uniformly bounded in $H^{1,2}(S)\cap L^p_\mu(\R^n)$.
By the fact that $L^p_\mu(\R^n)$ is a separable reflexive Banach space for any $1<p<\infty$ and  Lemma \ref{lem:gradient_uniqueness_current}, we obtain  that the spaces $H^{1,p}(S)$ are also Banach, seperable and reflexive. 
We therefore obtain a subsequence $k\to\infty$ (not relabled) and a function $u\in H^{1,2}(S)\cap L^p_\mu(\R^n)$ such that $u_k\weakto u$ in $H^{1,2}(S)$ and in $L^p_\mu(\R^n)$.
In particular, lower semicontinuity of the $L^2_\mu(R^n)$-norm implies that
\begin{equation}
  \int_{\R^n}|\nabla_{\mu} u|^2\d\mu \leq \liminf_{k\to\infty}\int_{\R^n}|\nabla_{\mu} u_k|^2\d\mu\,.
  \label{3eq:Imin1}
\end{equation}
In order to deduce lower semicontinuity of the term involving the double-well potential we employ the convergence properties deduced in the proof of Theorem \ref{thm:momolb}.

We let $S=\curr{\partial_*E,*\nu,1}$ and consider the generalized graphs $T_k:=T_{u_k,S}$.

Since $u_k\in\calD(\R^n)$ we obtain the representation 
\begin{equation*}
  T_k = (\Phi_k)_\# (S)\,,\quad
  \Phi_k(x):=(x,u_k(x))\text{ for }x\in \partial_* E\,.
\end{equation*}

We have
\begin{equation}\label{3eq:P-mass-bound}
  \begin{split}
    \mass{T_k} =\mass{(\Phi_k)_\# (S)}
    &=\int_{\partial_*E}\sqrt{1+|\nabla_{\mu}u_k|^2}\d\H^{n-1}\\
    &\leq \H^{n-1}(\partial_*E)+ \int_{\partial_* E} |\nabla_{\mu}u_k|^2\d\H^{n-1}\\
    &\leq \H^{n-1}(\partial_*E)+ \Lambda\,,
  \end{split}
\end{equation}
and
\begin{equation}
  \mass{T_{k,0}\ecke q} = \|u_k\|_{L^1(\mu)}
  \leq C(\Lambda,W,\mu(\R^n)).
  \label{3eq:P-L1-bound}
\end{equation}
This implies that the currents $T_k$ are uniformly bounded in mass and that $(u_k)_k$ is bounded in the desired $L^1$ sense.

We may use \cite[Theorem 4.3]{anzellotti1996bv} to conclude that there exists $v\in BV(S)$ such that
\[
  T_k \weakstarto T_{v,S}\,.
\]

By Proposition \ref{2pro:mfp-convergence} we also have the strong measure-function pair convergence
\begin{equation*}
  \big(\mu,u_k\big) \to \big(\mu,u\big)\quad\text{ in any }L^q,\,1\leq q<p\,,
  \label{3eq:mfp-v}
\end{equation*}
which implies (see \cite[Proposition 4.2.4]{hutchinson1986second}) that $u_k\to u$ in all $L^q_\mu(\R^n)$, $1\leq q<p$.
In particular, $\int_{\R^n} u\,d\mu=m$ holds.

Finally, choosing a sequence $(W_R)_R$ in $C^0_c(\R,\R^+_0)$ with $W_R\nearrow W$ as $R\to\infty$ we deduce by the Monotone Convergence Theorem and $W_R\leq W$
\begin{align*}
  \int_{\R^n} W(u)\d\mu = \lim_{R\to\infty} \int_{\R^n} W_R(u)\d\mu
  =\lim_{R\to\infty} \lim_{k\to\infty}\int_{\R^n} W_R(u_k)\d\mu
  \leq \liminf_{k\to\infty} \int_{\R^n} W(u_k)\d\mu\,.
\end{align*}
Together with \eqref{3eq:Imin1} this shows
\begin{equation*}
  I(u) \leq \liminf_{k\to\infty} I(u_k)\,,
\end{equation*}
which by the choice of $(u_k)_k$ proves the minimality of $u$.
\end{proof}

\section{Application to a two-phase biomembrane energy}
\label{sec:4}
In this section we consider a class of two-phase membrane energies that consist of a bending contribution given by a phase-dependent Willmore functional and a line tension energy.
Such kind of energies in particular appear as reductions of the Jülicher--Lipowsky energy discussed in the introduction.

The main result of this section connects diffuse and sharp interface description of such energies.

\begin{definition}[Two phase membrane energy, sharp interface formulation]
  \label{4eq:2pW}
Let $V\in \sfIV_2^o(\R^3)$ be an integer-rectifiable oriented $2$-varifold in $\R^3$ with weak mean curvature $H_V\in L^2_{\|V\|}(\R^3;\R^3)$ and consider a $\|V\|$-measurable function $u$ on $M=\supp(\|V\|)$ with $u(x)\in\{0,1\}$ for $\Ha^2$-almost every $x\in M$ and with $u\in BV(S)$, $S=\underline{c}(V)$.

We then define for given constants $a_1,a_2,k>0$ 
\begin{align}
  \W(u,V) &:= \frac{1}{4}\int_{\R^3} \big(a_1u + a_2(1-u)\big) |H_V|^2\d\|V\|\,,
  \label{4eq:defW}\\
  \calE(u,V) &:= \W(u,V) + 2k\Ha^1(J_u)\,,
  \label{4eq:Eeps}
\end{align}
where $J_u$ denotes the jump set of $u$ as introduced in Definition \ref{3def:Ju}.
\end{definition}

In the  corresponding diffuse interface description we will use the perimeter approximation from Section \ref{sec:3}.
In particular we assume a given nonnegative double-well potential $W$ with $\{W=0\}=\{0,1\}$ that satisfies \eqref{3eq:dw}.
We define $\dwp$ as in \eqref{eq:12}, the constant $k=\dwp(1)$ as in \eqref{3eq:k}, and the Modica--Mortola type functional $I_\eps$ as in \eqref{3eq:MM}.

Finally, we fix a smooth interpolation between the phase dependent constants $a_1,a_2$ from Definition \ref{4eq:2pW} of the form
\begin{equation}
  a^{\bar\omega}(r) = \bar\omega(r)a_1+(1-\bar\omega(r))a_2\,,
  \label{4eq:baromega}
\end{equation}
where $\bar\omega\in C^\infty_c(\R)$ satisfies $0\leq\bar\omega\leq 1 $ and $\bar\omega(0)=0$, $\bar\omega(1)=1$.

\begin{definition}[Two phase membrane energy, diffuse inteface description] 
Let $V_\eps\in \sfIV_2^o(\R^3)$ be an integer-rectifiable oriented 2-varifold in $\R^3$ with weak mean curvature $H_\eps:=H_{V_\eps}\in L^2_{\|V_\eps\|}(\R^3;\R^3)$, $S_\e:=\underline{c}(V_\e)$, and let $u_\eps\in H^{1,p}(S_\e)$ be given.
Then we define 
\begin{align}
  \W_{\bar\omega}(u_\eps,V_\eps) &:= \frac{1}{4}\int a^{\bar\omega}\circ u_\eps |H_\eps|^2\,d\|V_\eps\|,
  \label{4eq:Weps}\\
  \calE_\eps(u_\eps,V_\eps) &:= \W_{\bar\omega}(u_\eps,V_\eps) + I_\eps(u_\eps,S_\e)\,,
  \label{4eq:Eeps-2}
\end{align}
where $a^{\bar\omega}$ is as in \eqref{4eq:baromega} and $I_\eps$ as in Definition \ref{3def:MM}.
\end{definition}

Again we will consider in our main result a more restrictive setting, in particular to enforce the crucial strict convergence property that is needed in Theorem \ref{thm:momolb}.
Here we use the Li--Yau inequality \cite[Theorem 6]{li1982new}, see also \cite{KuSc04}, that guarantees for any $V\in \sfIV_2^o(\R^3)$ with $\W(V)<\infty$
that the two-dimensional density $\theta_2(\cdot,\|V\|)$ exists everywhere and satisfies
\begin{equation}\label{eq:10}
  \theta_2(\cdot,\|V\|)\leq \frac{\max(a_1^{-1},a_2^{-1})\W(u,V)}{4\pi}\,.
\end{equation}
From \eqref{eq:10} we deduce in particular that $V$ has unit density if $\W(u,V)<8\pi \min\{a_1,a_2\}$.

\begin{theorem}\label{4thm:main}
Consider a double-well potential $W$ with growth rate $p\geq 2$ be as in Theorem \ref{thm:momolb}.
Suppose $(E_\e)_\eps$ is a sequence of finite perimeter sets in $\R^3$ and let $\mu_\eps=\Ha^2\ecke \partial_*E_\eps$, $\nu_\eps:\partial_*E_\eps\to\sphere^2$ the inner normal, and $V_\eps=\ovar{\partial_*E_\eps,*\nu_\eps,1,0}$, $S_\e=\curr{\partial_*E_\e,*\nu_\e,1}$.

Consider in addition a sequence $(u_\eps)_\eps$ of phase fields $u_\eps\in H^{1,2}(S_\e)$.

Assume that for some $\Lambda,\delta>0$
\begin{align}
  \Ha^{n-1}(\partial_*E_\eps) + \calE_\eps(u_\e,V_\e) 
  &\leq \Lambda\quad\text{ for all }\eps>0\,,
  \label{4ass:1}\\
  \W_{\bar\omega}(u_\eps,V_\eps) &< 8\pi\min\{a_1,a_2\}-\delta\,.
  \label{4ass:2}
\end{align}
Then there exists a subsequence $\eps\to 0$, a finite perimeter set $E\subset\R^3$ and a function $u:\partial_*E\to \{0,1\}$ such that with $V=\ovar{\partial_* E,\nu_E,1,0}$, $S=\underline{c}(V)$ the following holds: $u$ belongs to $BV(S)$ and
\begin{align}
  V_\e&\wsto V\quad \text{ in }\sfIV^o_2(\R^3)\,,
  \label{4eq:thm2-1}\\
  (\mu_\eps,u_\eps)&\to (\mu,u)\quad\text{ as measure-function pairs in }L^q\,,
  \label{4eq:thm2-2}
\end{align}
for any $1\leq q<p$.

Moreover, it holds the lower bound estimate
\begin{equation}
  \liminf_{\e\to 0} \calE_\eps(u_\e,V_\e) \geq \calE(u,V)\,.
\end{equation}
\end{theorem}

The proof will be given below after some preparations.
\begin{remark}
In the proof we will obtain some additional properties.
In particular, we show the lower semicontinuity of both the bending and the line tension energy contribution to $\calE_\eps$ separately, and we obtain the strict convergence of the graphs of $\dwp\circ u_\eps$ to the generalized graph of $ku$.
\end{remark}

The next lemma is used to show the strict convergence of generalized graphs associated to $u_\eps$.
\begin{lemma}
\label{lem:multiplicity_one}
Let $V_\e\to V$ in $\sfIV^o_m(\R^n)$ such that
\[
  \begin{split}
    \theta_m(x,\|V\|)&=1\quad\text{ for }\|V\|\text{-almost every }x\in \R^n\,.
  \end{split}
\]
Then we have that 
\[
\mass{\underline{c}(V_\e)} \to \mass{\underline{c}(V)}\,.
\]
\end{lemma}

\begin{proof}
By \eqref{2eq:orvcurr} convergence as oriented varifolds implies convergence as currents, $\underline{c}(V_\e) \wsto \underline{c}(V)$.
The lower semicontinuity of the mass under convergence as currents yields
\begin{equation}
  \mass{\underline{c}(V)}\leq\liminf_{\e\to 0}\mass{\underline{c}(V_\e)}\,.
  \label{4eq:massb}
\end{equation}
Let us write
\[
  V_\e =\ovar{M_\e,\tau_\e,\theta_{\e,\pm}},\qquad
  V =\ovar{M,\tau,\theta_{\pm}}\,.
\]
By assumption 
\[
  \begin{split}
    \theta_++\theta_-
    =\theta_+-\theta_-=1\quad \|V\|\text{-almost everywhere.}
  \end{split}
\]
Since mass is continuous under varifold convergence we deduce
\[
  \begin{split}
    \limsup_{\e\to 0} \mass{\underline{c}(V_\e)}
    &\leq \limsup_{\e\to 0}\|V_\e\|(\R^n)\\
    &=\|V\|(\R^n)
    =\int_{M} (\theta_{+}+\theta_{-}) \d\H^m 
    =\int_{M} (\theta_{+}-\theta_{-}\big) \d\H^m 
    =\mass{\underline{c}(V)}\,.
  \end{split}
\]
Together with \eqref{4eq:massb} this completes the proof.
\end{proof}

\begin{proof}[Proof of Theorem \ref{4thm:main}]
By \eqref{4ass:1} we may pass to a further subsequence such that there exists  a set of finite perimeter $E\subset\R^3$ and an integral varifold $V\in \sfIV_2^o(\R^3)$ with 
\begin{align}
  \chi_{E_\e}&\to \chi_{E}\quad\text{ in }L^1(\R^3),\qquad
  \chi_{E_\e}\wsto \chi_{E}\quad\text{ in }BV\,,
  \label{4eq:Eeps-conv}\\
  V_\eps & \wsto V\quad\text{ as oriented varifolds,}\qquad
  \|V\| \geq |\nabla\chi_E|\,.
  \label{4eq:Veps-conv}
\end{align}
By \eqref{4ass:1} and the Li--Yau inequality \eqref{eq:10}, we have that $\|V\|$ almost everywhere $\theta_2(\cdot,\|V\|)=1$.
Hence $S:=\underline c(V)=[\![\partial_* E,*\nu_E,1]\!]$ and by Lemma \ref{lem:multiplicity_one} it follows that  
\[
  \lim_{\e\to 0}\H^{n-1}(\partial_* E_\e)=\lim_{\e\to 0}\mass{\underline c(V_\e)}=\mass{\underline c(V)}=\H^{n-1}(\partial_* E)\,,
\]
which shows the strict BV-convergence \eqref{3eq:BV-strict} of the sets $E_\eps$, $\eps>0$.

Therefore, we can apply Theorem \ref{thm:momolb} (see also its proof) and obtain the existence of some $u\in BV(S)$ with $u\in\{0,1\}$ $\|S\|$-almost everywhere such that 
\[
  T_{v_\eps,S_\e} \stackrel{c^*}{\wto} T_{\dwp(u),S}\,
\]
and such that the measure-function pair convergence \eqref{4eq:thm2-2} and the lower estimate
\[
  \liminf_{\e\to 0} I_\e(u_\e,S_\e)\geq 2k\H^{n-2}(J_u)
\]
hold.

It remains to show the lower semicontinuity statement for $\W$.
By $V_\e\wsto V$ as varifolds and by the uniform bound on $\|H_\e\|_{L^2_{\mu_\e}(\R^3;\R^3)}$, we  have the weak convergence of measure-function pairs
\begin{equation}\label{eq:15}
  (\|V_\e\|,H_\e)\wto  (\|V\|,H) \quad\text{ in }L^2\,.
\end{equation}
By \eqref{4eq:thm2-2} we further deduce the strong measure-function pair convergence
\begin{equation}\label{eq:16}
  \big(\|V_\e\|,\sqrt{a^{\bar\omega}(u_\e)}\big)
  \to \big(\|V\|,\sqrt{a^{\bar\omega}(u)}\big)\quad \text{ in }L^2\,.
\end{equation}
By \cite[Proposition 3.2]{moser2001generalization}, \eqref{eq:15} and \eqref{eq:16} may be combined to yield the weak convergence 
\begin{equation*}
  \Big(\|V_\e\|, H_\e \sqrt{a^{\bar\omega}(u_\e)}\Big)\wto \Big(\|V\|,H \sqrt{a^{\bar\omega}(u)}\Big) \quad\text{ in }L^1\,.
\end{equation*}
But clearly $\|H_\e \sqrt{a^{\bar\omega}(u_\e)}\|_{L^2_{\mu_\e}(\R^3;\R^3)}\leq \|\sqrt{a^{\bar\omega}}\|_{L^\infty}\|H_\e \|_{L^2_{\mu_\e}(\R^3;\R^3)}$ yields a uniform bound on the $L^2$ norms, and the weak convergence in $L^1$ can be upgraded to the weak convergence in $L^2$
\begin{equation}\label{eq:14}
  \big(\|V_\e\|, H_\e \sqrt{a^{\bar\omega}(u_\e)}\big)\wto \big(\|V\|,H \sqrt{a^{\bar\omega}(u)}\big) \quad\text{ in }L^2\,.
\end{equation}
By \cite[Theorem 4.4.2 (ii)]{hutchinson1986second} we obtain
\begin{align*}
  \W(u,V)&=\int \big(a_1u + a_2(1-u)\big)|H|^2\,d\|V\|\\
  &= \int |\sqrt{a^{\bar\omega}(u)}H|^2\,d\|V\|\\
  &\leq \liminf_{\eps\to 0}\int |\sqrt{a^{\bar\omega}(u_\eps)}H_\e|^2\,d\|V_\eps\|
  =\liminf_{\eps\to 0}\W_{\bar\omega}(u_\eps,V_\eps).
\end{align*}
This completes the proof of the theorem.
\end{proof}

Also in the case of the diffuse Jülicher--Lipowsky type energy for fixed $\eps>0$ we can provide an existence result, which however requires a rather strong smallness assumption on the energy.
Again, it is sufficient to consider the case $\eps=1$.

\begin{proposition}[Minimization of the diffuse Jülicher--Lipowksy energy]
Let a double-well potential $W$ with growth rate $p\geq 2$ be given as in Theorem \ref{thm:momolb}, let
\begin{equation}
  I(u,S)=\int_{\R^n} \big(|\nabla_{\mu} u|^2+W(u)\big) \d\mu\,,
  \label{4eq:MMFix}
\end{equation}
where $\mu=\|S\|$,
let $\W_{\bar\omega}$ be as in \eqref{4eq:Weps}, and set
\begin{align}
  \calE(u,V) &:= \W_{\bar\omega}(u,V) + I(u,\underline{c}(V))\,.
  \label{4eq:MinEeps-2}
\end{align}

For arbitrary $M_1,M_2>0$ consider
\begin{align*}
  X_2(M_2) &:= \big\{V=\ovar{\partial_* E,\nu_E,1,0}\,:\, \\
  &\qquad\qquad E\subset\R^3,\, E\text{ is a set of finite perimeter },\,\calH^2(\partial_*E)=M_2\big\}\,,
\end{align*}
and the class
\begin{equation*}
  X(M_1,M_2) := \Big\{(u,V)\,:\, V\in X_2(M_2),\, u\in H^{1,2}(\underline{c}(V))\,:\, \int_{\R^3} u\,\d\|V\| =M_1\Big\}\,.
\end{equation*}

Assume that for some $\delta>0$ and some $(u_0,V_0)\in X(M_1,M_2)$
\begin{align}
  \calE(u_0,V_0) &\leq 8\pi\min\{a_1,a_2\}-\delta\,.
  \label{4ass:2}
\end{align}
Then the functional $\calE$ attains its minimum in $X(M_1,M_2)$.
\end{proposition}

\begin{proof}
We again use the direct method of the Calculus of Variations and consider a minimal sequence $(u_k,V_k)_k$ in $X(M_1,M_2)$.
We let $V_k=\ovar{\partial_* E_k,\nu_k,1,0}$ with $E_k\subset\R^3$ a set of finite perimeter and $\nu_k$ generalized inner normal field of $E_k$, and write $S_k=\underline{c}(V_k)$.
Without loss of generality we may assume that
\begin{align}
  \calE(u_k,V_k) &\leq 8\pi\min\{a_1,a_2\}-\delta\quad\text{ for all }k\in\N\,.
  \label{4ass:2k}
\end{align}

As in the proof of Theorem \ref{4thm:main} we deduce by \eqref{4ass:2k}, the Li--Yau inequality \eqref{eq:10}, and Lemma \ref{lem:multiplicity_one} that there exist a subsequence $k\to\infty$ (not relabeled) and a set of finite perimeter $E\subset\R^3$ such that with $V=\ovar{\partial_* E,\nu_E,1,0}$ 
\begin{align}
  \chi_{E_k}&\to \chi_{E}\quad\text{ in }L^1(\R^3)\,,\qquad
  \chi_{E_k}\wsto \chi_{E}\quad\text{ in }BV(\R^3)\,,
  \label{4eq:MinEeps-conv}\\
  V_k & \wsto V\quad\text{ as oriented varifolds,}
  \label{4eq:MinVeps-conv}
\end{align}
and such that
\[
  M_2=\lim_{k\to\infty}\H^2(\partial_* E_k)=\H^2(\partial_* E)\,.
\]
The latter property shows the strict BV-convergence \eqref{3eq:BV-strict} of the sets $E_k$, $k\in\N$ and $V\in X_2(M_2)$.
We set $\mu=\|V\|=\calH^2\ecke\partial_*E$, $S=\underline{c}(V)$.

Following the proofs of Theorem \ref{thm:momolb} and Proposition \ref{3pro:MinMoMo} we further obtain the existence of some $u\in H^{1,2}(S)\cap L^p_\mu(\R^n)$   such that $u\in X(M_1,M_2)$ and
\[
  T_{u_k,S_k} \stackrel{c^*}{\wto} T_{u,S}\,\,
\]
and such that for any $1\leq q<p$ the measure-function pair convergence
\begin{align}
  (\mu_k,u_k)&\to (\mu,u)\quad\text{ as measure-function pairs in }L^q
  \label{4eq:MinThm2-2}
\end{align}
and the lower estimate
\begin{equation}
  \liminf_{k\to\infty} I(u_k,S_k)\geq I(u,S)
  \label{4eq:LiminfI}
\end{equation}
hold.

It remains to show the lower semicontinuity statement for $\W_{\bar\omega}$.
As in the proof of Theorem \ref{4thm:main}  we deduce 
\begin{equation}\label{eq:Min14}
  \big(\|V_k\|, H_k \sqrt{a^{\bar\omega}(u_k)}\big)\wto \big(\|V\|,H \sqrt{a^{\bar\omega}(u)}\big) \quad\text{ in }L^2\,.
\end{equation}
By \cite[Theorem 4.4.2 (ii)]{hutchinson1986second} we obtain
\begin{align*}
  \W_{\bar\omega}(u,V)&=\int |\sqrt{a^{\bar\omega}(u)}H|^2\,d\|V\|\\
  &\leq \liminf_{k\to\infty}\int |\sqrt{a^{\bar\omega}(u_k)}H_k|^2\,d\|V_k\|
  =\liminf_{k\to\infty}\W_{\bar\omega}(u_k,V_k).
\end{align*}
Together with \eqref{4eq:LiminfI} this yields that $(u,V)$ is a minimizer of $\calE$ in the class $X(M_1,M_2)$.
\end{proof}

\bibliographystyle{alpha}
\bibliography{phase}

\newcommand{\etalchar}[1]{$^{#1}$}
\begin{thebibliography}{DJB{\etalchar{+}}21}

\bibitem[ADS96]{anzellotti1996bv}
Gabriele Anzellotti, Silvano Delladio, and Giuseppe Scianna.
\newblock {$BV$} functions over rectifiable currents.
\newblock {\em Annali di Matematica Pura ed Applicata}, 170(1):257--296, 1996.

\bibitem[BBF99]{BeBF99}
Giovanni Bellettini, Guy Bouchitt\'{e}, and Ilaria Fragal\`a.
\newblock {$BV$} functions with respect to a measure and relaxation of metric
  integral functionals.
\newblock {\em J. Convex Anal.}, 6(2):349--366, 1999.

\bibitem[BBF01]{bouchitte2001convergence}
Guy Bouchitt{\'e}, Giuseppe Buttazzo, and Ilaria Fragal{\`a}.
\newblock Convergence of {S}obolev spaces on varying manifolds.
\newblock {\em Journal of Geometric Analysis}, 11(3):399--422, 2001.

\bibitem[BLS20]{BrLS20}
Katharina Brazda, Luca Lussardi, and Ulisse Stefanelli.
\newblock Existence of varifold minimizers for the multiphase
  {Canham–Helfrich} functional.
\newblock {\em Calc. Var. Partial Dif.}, 59(3):26, May 2020.
\newblock Id/No 93.

\bibitem[Bog07]{bogachev2007measure}
Vladimir Bogachev.
\newblock {\em Measure Theory, Vol. 1 and 2}.
\newblock Springer, 2007.

\bibitem[Can70]{Canh70}
P.B. Canham.
\newblock The minimum energy of bending as a possible explanation of the
  biconcave shape of the human red blood cell.
\newblock {\em J. Theor. Biol.}, 26(1):61--81, January 1970.

\bibitem[CMV13]{ChMV13}
Rustum Choksi, Marco Morandotti, and Marco Veneroni.
\newblock Global minimizers for axisymmetric multiphase membranes.
\newblock {\em ESAIM: Control, Optimisation and Calculus of Variations},
  19(4):1014--1029, July 2013.

\bibitem[DJB{\etalchar{+}}21]{DJBS21}
Yannik Dreher, Kevin Jahnke, Elizaveta Bobkova, Joachim~P. Spatz, and Kerstin
  Göpfrich.
\newblock {Division and Regrowth of Phase-Separated Giant Unilamellar
  Vesicles}.
\newblock {\em Angewandte Chemie International Edition}, 60(19):10661--10669,
  2021.

\bibitem[Hel73]{Helf73}
W~Helfrich.
\newblock Elastic properties of lipid bilayers: {Theory} and possible
  experiments.
\newblock {\em Zeitschrift für Naturforschung C}, 28(11-12):693--703, December
  1973.

\bibitem[Hel12]{Helm13}
Michael Helmers.
\newblock Kinks in two-phase lipid bilayer membranes.
\newblock {\em Calc. Var. Partial Dif.}, 48(1-2):211--242, August 2012.

\bibitem[Hel14]{Helm15}
M.~Helmers.
\newblock Convergence of an approximation for rotationally symmetric two-phase
  lipid bilayer membranes.
\newblock {\em Q. J. Math.}, 66(1):143--170, October 2014.

\bibitem[Hut86]{hutchinson1986second}
John~E Hutchinson.
\newblock Second fundamental form for varifolds and the existence of surfaces
  minimising curvature.
\newblock {\em Indiana University Mathematics Journal}, 35(1):45--71, 1986.

\bibitem[JL96]{JuLi96}
Frank Jülicher and Reinhard Lipowsky.
\newblock Shape transformations of vesicles with intramembrane domains.
\newblock {\em Phys. Rev. E}, 53(3):2670--2683, March 1996.

\bibitem[KS04]{KuSc04}
Ernst Kuwert and Reiner Schätzle.
\newblock Removability of point singularities of willmore surfaces.
\newblock {\em Ann. Math.}, 160(1):315--357, July 2004.

\bibitem[LY82]{li1982new}
Peter Li and Shing-Tung Yau.
\newblock A new conformal invariant and its applications to the {W}illmore
  conjecture and the first eigenvalue of compact surfaces.
\newblock {\em Inventiones mathematicae}, 69(2):269--291, 1982.

\bibitem[Men16a]{menne2016sobolev}
Ulrich Menne.
\newblock Sobolev functions on varifolds.
\newblock {\em Proceedings of the London Mathematical Society},
  113(6):725--774, 2016.

\bibitem[Men16b]{menne2016weakly}
Ulrich Menne.
\newblock Weakly differentiable functions on varifolds.
\newblock {\em Indiana University Mathematics Journal}, pages 977--1088, 2016.

\bibitem[MM77]{MoMo77}
Luciano Modica and Stefano Mortola.
\newblock Un esempio di {$\Gamma$}-convergenza.
\newblock {\em Boll. Un. Mat. Ital. B (5)}, 14(1):285--299, 1977.

\bibitem[Mod87]{Modi87}
Luciano Modica.
\newblock The gradient theory of phase transitions and the minimal interface
  criterion.
\newblock {\em Arch. Ration. Mech. An.}, 98(2):123--142, June 1987.

\bibitem[Mos01]{moser2001generalization}
R.~Moser.
\newblock A generalization of {R}ellich’s theorem and regularity of varifolds
  minimizing curvature.
\newblock Preprint~72, MPI MiS Leipzig, 2001.

\bibitem[Oss97]{Ossa97}
Elisabetta Ossanna.
\newblock {$SBV$} functions over a rectifiable current and a compactness
  theorem.
\newblock {\em J. Math. Pures Appl. (9)}, 76(1):1--13, 1997.

\bibitem[Res68]{reshetnyak1968weak}
Yu~G Reshetnyak.
\newblock Weak convergence of completely additive vector functions on a set.
\newblock {\em Siberian Mathematical Journal}, 9(6):1039--1045, 1968.

\bibitem[Sim83]{Simo83}
Leon Simon.
\newblock {\em Lectures on geometric measure theory}, volume~3 of {\em
  Proceedings of the Centre for Mathematical Analysis, Australian National
  University}.
\newblock Australian National University, Centre for Mathematical Analysis,
  Canberra, 1983.

\bibitem[Spe11]{spector2011simple}
Daniel Spector.
\newblock Simple proofs of some results of {R}eshetnyak.
\newblock {\em Proceedings of the American Mathematical Society},
  139(5):1681--1690, 2011.

\end{thebibliography}
 
\end{document}